\documentclass[reqno]{amsart}
\usepackage{amsmath, color}
\usepackage{amssymb}
\usepackage{amsfonts}
\usepackage{graphicx}
\newtheorem{thm}{Theorem}[section]

\newtheorem{lem}{Lemma}[section]

\newtheorem{prop}{Proposition}[section]
\newtheorem{definition}{Definition}[section]
\newtheorem{remark} {Remark}[section]
\numberwithin{equation}{section}

\newcommand{\ba}{\begin{aligned}}
\newcommand{\ea}{\end{aligned}}
\newcommand{\be}{\begin{equation}}
\newcommand{\ee}{\end{equation}}
\newcommand{\bnn}{\begin{eqnarray*}}
\newcommand{\enn}{\end{eqnarray*}}
\newcommand{\n}{\rho}
\newcommand{\ep}{\epsilon}

\newcommand{\g}{\gamma}
\newcommand{\p}{\partial}

\renewcommand{\div}{{\rm div}}
\newcommand{\na}{\nabla}
\newcommand{\la}{\label}

\def\u{{\bf u}}
\def\norm[#1]#2{\|#2\|_{#1}}
\def\m{{\bf m}}
\def\o{\Omega}
\def\de{\delta}

\def\th{\theta}

\def\r{\mathbb{R}^{n}}
\def\u{{\bf u}}
\def\lap{\triangle}

\begin{document}

\title[Long-time behavior     for NS equations]{Long-time behavior of weak solutions    for compressible  Navier-Stokes equations with degenerate viscosity}%
\author[Z. Liang]{Zhilei Liang}
\address{School of   Mathematics, Southwestern  University of Finance and Economics, Chengdu  611130,  China}
\email{zhilei0592@gmail.com}
\date{}
\begin{abstract}   The long-time regularity and   asymptotic of  weak solutions are studied for    compressible Navier-Stokes equations with degenerate  viscosity in a bounded periodic domain in   two and three  dimensions.   It is shown  that the density keeps strictly positive from below and above  after  a  finite  period of time.  Moreover, higher velocity regularity    is obtained via a parabolic type  iteration technique. Since then  the weak solution  conserves its energy equality,  and    decays  exponentially  to the   equilibrium in  $L^{2}$-norm  as time goes to infinity. In addition, assume that  the initial momentum is zero, the  exponential   decay  rate of the derivatives is derived,   and  the weak solution  becomes a strong one in two dimensional space.
\end{abstract}
\maketitle
\section{Introduction}

The  time-evolutionary Navier-Stokes equations simulate the    motion of un-stationary   compressible
fluids. It is an  important mathematical model in continuous medium mechanics theory and has a wide applications   in many fields, such as
astrophysics, engineering, and so on.  In this paper we focus on the     isentropic compressible Navier-Stokes equations (cf. \cite{lau,p2})
\begin{equation}\label{n1}
\left\{\ba
& \p_{t}\n+{\rm div}(\n \u)=0,\\
& \p_{t}(\n \u) +{\rm div}(\n \u \otimes \u)+\na \n^{\g}=\div(\n \na\u).\ea \right.
\end{equation}
Here, variables $t>0,\,x\in\o  \subseteq \r,\, (n=2,3), $   the   unknown functions $\n(x,t)$ and $\u(x,t)$ are the density and the velocity field  respectively;    the adiabatic exponent $\g$ is assumed to satisfy $\g>1.$

In this paper we are limited to 
the bounded domain with periodic boundaries, i.e.,  $\o=\mathbb{T}^{n}$.  For completeness, we    supplement equations \eqref{n1} with  the initial functions
\be\la{n1a}\n(x,0)=\n_{0}\ge 0,\quad \n \u (x,0)=\m_{0},\quad x\in\o.\ee

\bigskip

There are many literatures   concerning  the  existence, regularity, and long-time behavior of    solutions for compressible Navier-Stokes equations.  We begin with the constant viscous coefficient case and   collect  some  existence and regularity  results  of weak solutions, without   completeness.  The progress of one-dimensional problem is satisfactory, see, e.g., \cite{anton,hoff,jiang1,jiang2,liangli}. For high dimensions with positive density, we refer to the papers by Hoff \cite{hoff1} and by Serre \cite{serr}. However, it  is a very challenging
problem in mathematics when initial vacuum appears,    Lions \cite{p2} first   proved the global existence of weak solutions when the  adiabatic exponent  $\g\ge\frac{9}{5}$.    Feireisl-Novotn\'{y}-Petzeltov\'{a} \cite{fei} improved  Lions'  result  to a wider range of $\g>\frac{3}{2}$.  Recently, Plotnikov-Weigant \cite{pw} discussed the global existence of weak solutions in   case of $\g=\frac{3}{2}$.  Jiang-Zhang \cite{jiangzhang} obtained  the existence of weak solutions  for  all  $\g>1$ if   some   symmetry assumptions are made. As far as the   regularity of   solutions is concerned,   Desjardins \cite{des} proved that if  $\g>3$,   the weak solutions satisfy, for some $T>0,$
\bnn\ba &\n\in L^{\infty}\left(0,T;L^{\infty}(\o)\right),\quad \na \u\in L^{\infty}\left(0,T;L^{2}(\o)\right),\\
& \sqrt{\n}\u_{t},\,\na\times (\na\times \u),\,\na(\div \u-P)\in L^{\infty}\left(0,T;L^{2}(\o)\right).\ea\enn
  Choe-Jin \cite{kim} discussed similar   regularity of weak solutions  in case of zero   bulk viscosity.

Physically,  the dynamics of the viscous fluids near vacuum are better modeled by the
  Navier-Stokes equations with density-dependent viscosities. It can also be understand mathematically in the
derivation of the compressible Navier-Stokes equations from the Boltzmann equation
by the Chapman-Enskog expansions, where  the viscosity depends on the temperature and thus the  density  for isentropic flows.  Equations \eqref{n1}   corresponds to the shallow water
model in the case  $\g=2$ in dimension two,  where $\n$ stands for the height of the water.  Such model describes  the horizontal structure of the fluids, and appear   often in geophysical flows (cf.\cite{g,gp,p2}).  We also regard  \eqref{n1}  as  a special case of the   equations
   \begin{equation}\label{1.3}
\left\{\ba
& \p_{t}\n+{\rm div}(\n \u)=0,\\
& \p_{t}(\n \u) +{\rm div}(\n \u \otimes \u)+\na \n^{\g}=\div \mathbb{S}, \ea \right.
\end{equation} where the stress tensor
  \bnn \mathbb{S} =h(\n)\na \u +g(\n) \div \u \mathbb{I} \quad {\rm or}\quad \mathbb{S}=h(\n)\mathbb{D} \u+ g(\n)\div \u \mathbb{I},\enn
$\mathbb{I}$ is the identical matrix in $\r$.

For smooth solutions to equations \eqref{1.3} on condition that  $ g(\n)=\n h^{'}(\n)-h(\n)$,   Bresch-Desjardins \cite{bres,bres1}  first  developed
a new entropy estimate
\be\la{1.4} \ba &\frac{d}{dt}\int\left(\frac{1}{2}\n|\u+\na \varphi(\n)|^{2}+\frac{1}{\g-1}\n^{\g}\right)\\
&+\int\left(\na \varphi(\n)\na\n^{\g}+h(\n)|\na\u-(\na\u)^{tr}|^{2}\right)=0,\quad \n\varphi'=h'.\ea\ee
Li-Li-Xin \cite{llx} studied the one-dimensional problem, they established   the global  entropy weak solution to  \eqref{1.3} and discussed the long-time dynamics:  vanishing  vacuum   states  and
blow-up phenomena. Stra\u{s}kraba-Zlotnik \cite{sz} studied the global regularity of weak solution and derived exponential  decay rate estimates.   We also refer   to  the papers \cite{has,jiang3,llx,kv,yang1,yang2,yang3} for    related   results in dimension one,  and the  papers   \cite{guo,jiu} for  high-dimensions  with     symmetric assumptions.
For general high dimensions $n=2,3,$ Mellet-Vasseur \cite{mv} provided a compactness framework so that  the   weak solutions can be established  from a sequence of  smooth   approximate  solutions.   On the basis of
 entropy estimate \eqref{1.4} due to Bresch-Desjardins,  as well as Mellet-Vasseur type estimates in \cite{mv}, the global existence of weak solutions to the problem \eqref{n1}-\eqref{n1a} are derived by  Li-Xin  \cite{lx} and  Vasseur-Yu \cite{yu} from different approach.

By weak solutions, we mean
\begin{definition} \la{defi} Function $(\n,\u)$ is called a weak solution  to  the problem \eqref{n1}-\eqref{n1a}  if for any fixed $T>0$
\be\la{1.5}\begin{cases}&0\le \n\in L^{\infty}\left(0,T; L^{1}(\o)\cap L^{\g}(\o)\right),\\
& \na  \sqrt{\n} \in L^{\infty}\left(0,T; L^{2}(\o)\right),
\quad
 \na \n^{\frac{\g}{2}} \in L^{2}\left(0,T; L^{2}(\o)\right)\\
 & \sqrt{\n}\u\in L^{\infty}\left(0,T; L^{2}(\o)\right),\quad \sqrt{\n} \na\u\in L^{2}\left(0,T; W^{-1,1}(\o)\right) .
\end{cases} \ee
Moreover,

 (i).  the problem \eqref{n1}-\eqref{n1a} is satisfied in the sense of  $\mathcal{D}'(\o\times [0,T)),$

(ii).  for  a.e. $t>0$,  the energy    inequality
\be\la{1.6} \frac{d}{dt}\int_{\o}\mathcal{E}(x,t)dx+ \int_{\o}\Lambda^{2}dx\le 0   \ee
and the entropy inequality \be\la{1.7}  \frac{d}{dt}\int_{\o}\n\left|\u+\na \ln\n\right|^{2}dx+  \int_{\o}\left(|\na \n^{\frac{\g}{2}}|^{2}+\Lambda^{2}\right)dx\le 0  \ee  are fulfilled,
where  the energy density   \be\la{1.8} \mathcal{E}  = \left(\frac{1}{2}\n |\u|^{2}+\frac{1}{\g-1}\n^{\g}\right)  \ee
and the function $\Lambda\in L^{2}(\o\times (0,T))$ satisfying \be\la{1.9}\ba \int_{0}^{T}\int_{\o} \Lambda   \Phi =-\int_{0}^{T}\int_{\o}\left(\sqrt{\n}\sqrt{\n}\u \div \Phi+2\sqrt{\n}\u  \Phi \na\sqrt{\n}\right) \ea\ee for all  smooth functions $\Phi$ valued in $\mathbb{R}^{n\times n}.$
\end{definition}

The regularity of weak solutions to \eqref{n1} is  an important and interesting question.
 We remark that the \emph{inequality sign} $"\le"$ in \eqref{1.6}  means   anomalous dissipation of energy, which  is mainly caused by   the wildness of the weak
solutions.  The energy dissipation is one of basic properties of the fluid equations  related
to its physical origin.  It is motivated  from Kolmogorov's theory of turbulence of flow and also reminiscent of the Leray-Hopf weak solutions to
the incompressible Navier-Stokes equations.  Clearly,  if the  solution is smooth enough, then  energy equality conserves automatically.    Yu   \cite{yu2} obtained a sufficient condition so that  the weak solutions to the equations  \eqref{n1} preserve  energy conservation  for all positive time. The main idea in \cite{yu2} is using  the  commutator estimates developed by  DiPerna-Lions \cite{di} to deal with the nonlinear term $\p_{t}(\n\u)$.    For other regularity criterion for energy conservation of weak solutions for  compressible fluids, we mention  the papers \cite{liang,liang1,wang} and the references cited therein.

\bigskip

\subsection{Main results}
The main concern of this current article is   the regularity   and the  long-time asymptotic  of the weak solutions to the problem \eqref{n1}-\eqref{n1a}.

Our main results read in the  theorems below.
\begin{thm} \la{t1}  Let   $\o$ be  a bounded domain with periodic boundaries  in $\r$, i.e., $\o=\mathbb{T}^{n}$ with   $n=2,3$. Let   $(\n,\u)$ be a weak solution to  the  problem \eqref{n1}-\eqref{n1a} in   Definition \ref{defi}.

 Then there exist  positive  constants   $\underline{C}$,   $\overline{C}$, and  some large time point  $T<\infty$,  which depend only on the initial functions and $\o$,   such that
 \be\la{1.10} 0< \underline{C}  \le   \n(x,t)\le \overline{C}<\infty,\quad a.e.\,\,\,(x,t)\in \o\times (T,\infty)\ee
 and \be\la{1.11} \u\in L^{\infty}\left(T,\infty; L^{\infty}(\o)\right),\quad \sqrt{\n} \na\u\in L^{2}\left(T,\infty; L^{2}(\o)\right).\ee
Furthermore,  the following  assertions are valid:

(i).(Energy conservation) For almost all  $ t_{1}>T,$ and all  $t \in [t_{1},\infty)$, the  solution $(\n,\u)$ to \eqref{n1}-\eqref{n1a} keeps energy conservation, i.e.,
 \be\la{1.12} \int_{\o}\mathcal{E}(x,t)dx+ \int_{t_{1}}^{t}\int_{\o}\n|\na\u|^{2}dxdt  =\int_{\o}\mathcal{E}(x,t_{1})dx, \ee
 where $ \mathcal{E}$ is defined in \eqref{1.8}.

 (ii)(Exponential asymptotic) There exist  constants $\sigma$ and   $C$ which   rely only  on the initial functions and $\o$, such that
\be\la{1.13}\|\u-(\u)_{a}\|_{L^{2}}^{2} +  \| \n-\overline{\n}\|_{L^{2}}^{2} \le Ce^{-\sigma t}, \quad t> T,\ee
where
\be\la{1.14} \overline{\n}=\frac{\int_{\o} \n_{0}(x)dx}{|\o|}>0\quad {\rm and}\quad (\u)_{a}=\frac{\int_{\o} \m_{0}(x)dx}{\int_{\o} \n_{0}(x)dx}.\ee  \end{thm}

\bigskip

 \begin{remark} Inequality  \eqref{1.10} shows that no concentration or     vacuum state  can be formed  after a finite period of  time, regardless of   what the initial state  is.  Moreover, the conservation of energy for weak solutions is automatically established.   \end{remark}

In the second theorem we derive   higher  regularity of  the weak solutions, if   the initial  momentum is  assumed to be zero.

\begin{thm} \la{t2} In addition to the hypotheses  in Theorem \ref{t1},   we assume that  the initial momentum is zero, namely,
\be\la{1.15} \left|\int_{\o} \m_{0}(x)dx\right|=0. \ee

 Then, for some $T_{1}\in (T,\infty)$,  there exist  constants  $\sigma_{1}\in (0,\sigma)$  and  $C$  depending     only  on the initial functions and $T_{1},\,\o,$  such that
\be\la{1.16}\|\na  \n \|_{L^{2}}^{2} \le Ce^{-\sigma_{1} t}, \quad t> T_{1}.\ee

Particularly,   $(\n,\u)$ becomes a   strong solution in two-dimensional space, satisfying the improved velocity regularity
 \be\la{1.17} \ba &\u\in L^{\infty}\left(T_{1},\infty;   H^{1}(\mathbb{T}^{2})\right) \cap L^{2}\left(T_{1},\infty; W^{2,\frac{3}{2}}(\mathbb{T}^{2})\right),\\
&\quad \qquad \u_{t}\in L^{2}\left(T_{1},\infty; L^{2}(\mathbb{T}^{2})\right), \ea\ee
and the long-time asymptotics
\be\la{1.18}\|\na  \n \|_{L^{2}(\mathbb{T}^{2})}^{2}+ \|\na \u \|_{L^{2}(\mathbb{T}^{2})}^{2} \le Ce^{-\sigma_{1} t}, \quad t> T_{1}.\ee \end{thm}

   \bigskip

 \begin{remark} Regularity \eqref{1.17} implies that $(\n,\u)$  satisfies the equations \eqref{n1} for almost everywhere in $\o\times (T_{1},\infty)$ in two-dimensional space.  In this regard,  we call $(\n,\u)$  a   strong solution. However,   whether the regularity \eqref{1.17} guarantees the   uniqueness  or not is unclear.   \end{remark}

 \begin{remark}  By slight modification,   Theorems   \ref{t1}-\ref{t2}   work  for general bounded domain with suitable boundary conditions, as long as the  existence of weak solutions is known.\end{remark}

\bigskip

\subsection{Methodology} Let us give a brief,  heuristic overview of the proof and explain  the underlying   motivations.

\bigskip

We are mainly motivated from the one-dimensional results in \cite{llx,sz}, the global existence of weak solutions in \cite{lx,yu}, and the regularity criterion in \cite{des,liang,yu2}.

We first  approximate the density in \eqref{3.6}  by $\n_{k}$   so that for   large number $k$ \be\la{1.19} \left[\frac{3}{4},\, \frac{5}{4}\right]\ni \overline{\n_{k}}=\frac{1}{|\o|}\int_{\o}\n_{k}(x,t)dx.\ee
(In the proof  we assume that $\frac{1}{|\o|}\int_{\o}\n_{0}(x)dx=\int_{\o}\n_{0}(x)dx=1$.) We remark that the  approximate function $\n_{k}$    is to
  avoid  the possible  concentration and  vacuum. Notice that $\na\n \in L^{\infty}\left(0,T;L^{2}\right)$, we have  $\p \n_{k}=  \p \n$ in  $\left\{(x,t):\, \frac{1}{k}\le \n\le k\right\}$ and $\p \n_{k}= 0$ for others. Therefore,    $ \na  \n_{k}  \in L^{\infty}(0,\infty;L^{2}(\o))\cap L^{2}(0,\infty;L^{2}(\o)).$
A careful computation shows that,  for any fixed   $k<\infty,$   \be\la{1.20} \lim_{t\rightarrow \infty}\|\n_{k}-\overline{\n_{k}}\|_{L^{\infty}(\o)}=0.\ee
By  \eqref{1.20} we see   that  $\n_{k}(x,t)$ stablizes to its average value $\overline{\n_{k}}$ when   the  time $t$ goes to infinity. In particular,  there is some finite $T^{*}=T^{*}(k)<\infty$ such that
\bnn \|(\n_{k}-\overline{\n_{k}})(\cdot,t)\|_{L^{\infty}(\o)}\le \frac{1}{4},\quad{\rm for\,\,all}\,\, t>T^{*}.\enn
This, along with  \eqref{1.19}, implies that for $t>T^{*}$
\bnn   \frac{1}{2}= \frac{3}{4}-\frac{1}{4}\le  \overline{\n_{k}}-\frac{1}{4}\le \n_{k}\le  \overline{\n_{k}}+\frac{1}{4}\le \frac{5}{4}+\frac{1}{4}= \frac{3}{2},\enn that is, \be\la{1.21}\frac{1}{2}\le \n_{k}\le \frac{3}{2},\quad a.e.\,\,\, (x,t)\in\o\times (T^{*}, \infty).\ee
Remember that \eqref{3.6},  for large $k<\infty$ we deduce from \eqref{1.21} that
\bnn\frac{1}{2}\le \n\le \frac{3}{2},\quad a.e.\,\,\, (x,t)\in\o\times (T^{*}, \infty).\enn  Similar argument runs for the general case, and thus  \eqref{1.10} is proved.

Once \eqref{1.10} is obtained, the momentum equations can be regarded as a parabolic system in terms of velocity $\u$. Hence, we are motivated to expolit   a parabolic type    iteration technique as well as a  continuity method to conclude  \eqref{1.11}.  The former part in   \eqref{1.11} directly yields $\u\in L^{4}_{loc}(T,\infty;L^{6}(\o))$.  This and  \eqref{1.10}   enable us to deduce  that    the weak solution conserves it energy equality,  following  the same step as   in \cite{yu2}.   Finally,  by means of  the  inverse operator estimates  and basic energy estimates, we show that   the  solution   decays exponentially to its   equilibrium in the sense of  $L^{2}(\o)$ topology   as time goes to infinity.

\smallskip

 In Theorem \ref{t2}, we  obtain better regularity if the initial momentum is assumed to be zero. In particular, in case of
 \eqref{1.15}, we first deduce from  \eqref{1.13} that   $e^{\sigma t}\|\u\|_{L^{2}}^{2}  \le C$. By this and the   entropy inequality \eqref{1.7} we successfully proved that   $\|\na \n\|_{L^{2}}$ has  an exponential decay rate as  time tends to infinity.  Next, in deriving  the higher derivative estimates, we need to overcome the difficulty resulting from  the density appeared  in the diffusion.  An important observation is  the    Sobolev embedding in dimension two, i.e.,    $\|\na^{2} \u\|_{L^{\frac{3}{2}}}\hookrightarrow \|\na \u\|_{L^{6}}.$ This together with the time decay  of $\|\na \n(\cdot,t)\|_{L^{2}}$ guarantee    higher velocity regularity    in two dimensional case, and consequently,  an  exponential decay  of  $L^{2}$-norm  of the velocity  derivative.

\bigskip

 \textbf{Notation:}   Refer to \cite{adams},   we denote   the standard  Sobolev spaces   by
\bnn\ba  L^{p}=L^{p}(\o),\quad  W^{k,p}=\{f:\,\,|\p^{i}  f|\in L^{p},\,\,0\le i\le k\},\quad H^{k}=W^{k,2},
\ea\enn
and  use the simplified conventions
  \bnn \int f=\int_{\o}f(x)dx,\quad  \overline{f}=\frac{1}{|\o|}\int f,\enn
  where  $|\o|$  is  the Lebesgue measure of $\o$.
Throughout this paper, the capital letters $C ,\, C_{i} \,(i=1,2,..)$ symbol  positive   constants  which may vary from line to line.  In addition,   $C(a)$ is used to emphasize that $C$ depends on $a.$

\bigskip

\section{Preliminaries}

The first lemma is responsible for  the existence results   of weak solutions to the problem \eqref{n1}-\eqref{n1a}.
\begin{lem}[See \cite{lx,yu}]\la{lem2.1} For any $T<\infty,$ the problem \eqref{n1}-\eqref{n1a} admits a weak solution $(\n,\u)$ over $(0,T)$ in the sense of Definition \ref{defi}, which satisfies \eqref{1.5}-\eqref{1.7}.
 \end{lem}

 The following  embedding inequalities   will be frequently used throughout this paper.
 \begin{lem}[Gagliardo-Nirenberg]\cite{adams,lad} \la{lem2.2} Let $\Omega\subseteq\mathbb{R}^{n}$ be a  bounded   domain with   smooth boundaries. It holds that  for any $v\in W^{1,q}\cap L^{s}$
\be\la{2.1}\ba \|v\|_{L^{p}}\le C_{1}\|v\|_{L^{s}}+ C_{2}\|\na v\|_{L^{q}}^{\th}\|v\|_{L^{s}}^{1-\th},\ea\ee where   the constant $C_{i} (i=1,2)$  depends  only  on $p,q,s, \th,$    the exponents satisfy \bnn 0\le \th\le 1,\quad 1\le q,\,s\le \infty ,\quad \frac{1}{p}=\th(\frac{1}{q}-\frac{1}{n})+(1-\th)\frac{1}{s}\enn  and
\bnn\left\{
  \begin{array}{ll}
\min\{s,\,\frac{nq}{n-q}\}\le p\le \max\{s,\,\frac{nq}{n-q}\}, & {\rm if}\,\,q<n; \\
s\le p<\infty, &  {\rm if}\,\,q=n;\\
s\le p\le \infty, &  {\rm if}\,\,q>n.
  \end{array}
\right.\enn
Moreover, $C_{1}=0$  if  0-Dirichlet boundary condition or zero average is assumed.

As an application of \eqref{2.1}, we have
\be\la{2.2}\ba \|v-\overline{v}\|_{L^{p}}\le  C \|\na v\|_{L^{\frac{np}{n+p}}} ,\quad p\in [1,6]\ea\ee
and
\be\la{2.3} \ba  \left( \int_{-r}^{r}\int|v^{2}|^{\frac{5}{3}}dxdt\right)^{\frac{3}{5}} \le C\left(\sup_{t\in [-r,r]}\int |v|^{2}dx+ \int_{-r}^{r}\int \left(|v |^{2}  +|\na v|^{2}\right)dxdt\right).\ea \ee
 \end{lem}
 \begin{proof} We only prove  \eqref{2.3}.   In fact, by \eqref{2.1} and h\"{o}lder inequality,
 \bnn\ba & \int_{-r}^{r}\int|v^{2}|^{\frac{5}{3}}dxdt \\
 &\le \sup_{t\in [-r,r]}\left(\int|v^{2}|^{\frac{2}{3}}dx \right) \int_{-r}^{r}\int|v^{2}|dxdt \\
 &\le C \sup_{t\in [-r,r]}\left(\int|v^{2}|dx\right)^{\frac{2}{3}} \int_{-r}^{r} \left(\int|v|^{6}dx \right)^{\frac{1}{3}}dt \\
  &\le C \sup_{t\in [-r,r]}\left(\int|v^{2}|dx\right)^{\frac{2}{3}} \int_{-r}^{r} \int\left(|v |^{2}  +|\na v|^{2}\right)dxdt. \ea\enn
Raising the above expression  to the power of $\frac{3}{5}$, applying the Young  inequality, we obtain \eqref{2.3}.
 \end{proof}

 \begin{lem}\la{lem2.3} Let $\o\subset \r$ be   a bounded domain with smooth boundaries.  Then, for all  $v\in W^{k,p}(\o)\cap W^{1,p}_{0}(\o)$ with $p\in (1,\infty)$ and integer $k\ge 0$,  there    exists a positive constant
$C$ which  depends  only on $p,\,n,\,k$ such that
\be\la{p50}\|\na^{k+2}v\|_{L^{p}}\le C\|\lap v\|_{W^{k,p}}\ee
 \end{lem}
 \begin{proof} The proof is a classical elliptic regularity theories in \cite{ag}. See also \cite[Lemma 12]{kim1}. \end{proof}

\begin{lem}[Bogovskii]\la{lem2.4}
Let $\o\subset \r$ be a bounded Lipschitz domain, and let $p\in (1,\infty)$. There is a linear operator $\mathcal{B}= (\mathcal{B}^{1},...,\mathcal{B}^{n})$ from $L^{p}$ to $W_{0}^{1,p}$ such that for all $v\in L^{p}$ with  $\overline{v}=0,$
  \bnn \div \mathcal{B}(v)=v\quad a.e.\,\,{\rm in}\,\,\o\quad {\rm and}\quad \|\na \mathcal{B}(v)\|_{L^{p}}\le C(p,n,\o)\|v\|_{L^{p}}.\enn\end{lem}
\begin{proof} The detailed proof is available in \cite{nov}.\end{proof}

\bigskip

\section{Proof of Theorem \ref{t1}}

First,  for fixed $t>0$, we  define
\bnn \psi(s)=\left\{\ba &1,& s\le t,\\
&k(t+\frac{1}{k}-s),&t\le s\le t+\frac{1}{k},\\
&0,&s\ge t+\frac{1}{k}.\ea\right.\enn
Since $\o=\mathbb{T}^{n}$ is periodic,  we test the mass equation  $\eqref{n1}_{1}$ against  $\psi(t)$ and receive
 \bnn  \int_{t}^{t+\frac{1}{k}}\int\n(x,s)=\int\n_{0}.\enn
By  Lebesgue point theorem,    sending $k\rightarrow \infty$ in above equality yields
 \be\la{3.1}\int \n(x,t)=\int \n_{0}(x),\quad a.\,e.\,\,\,\,t>0.\ee

Similar argument runs that  \be\la{3.2}\int \n\u(x,t)=\int \m_{0}(x),\quad a.\,e.\,\,\,\,t>0.\ee

\subsection{Positive bounds of density} We will show  that  the density is positively  bounded from below  after a finite time interval.   Without loss of generality,   we  assume
\be\la{3.3} |\o|=1\quad {\rm and} \quad  \int_{\o}\n_{0}=1,\ee
and prove that, for some large $T^{*}<\infty,$
 \be\la{3.4} \frac{1}{2}\le \n \le \frac{3}{2},\quad a.e.\quad (x,t)\in \o\times (T^{*},\infty).\ee
For more general case,   the same deduction  yields     \eqref{1.10}.

\bigskip

The main task of this subsection is to justify \eqref{3.4}.

\bigskip

By virtue of Lemma \ref{lem2.1},  the weak solution $(\n,\u)$  satisfies  \eqref{1.5}, that is,
 \be\la{3.5}\ba &\sup_{t\ge0}\left\|\na \sqrt{\n}\right\|_{L^{2}}+\int_{0}^{\infty}  \left( \left\|\na \n^{\frac{\g}{2}}\right\|_{L^{2}}^{2}+\| \sqrt{\n}\na \u \|_{L^{2}}^{2}\right)dt \le C.\ea\ee
Here, and  during  this subsection,  the constant $C$ is  independent of $t.$

 Let us approximate the density $\n$ by
\be\la{3.6}  \n_{k} =\left\{\ba &k,&k\le \n,\\
&\n,& \frac{1}{k}\le  \n\le k,\\
&\frac{1}{k},& \n\le \frac{1}{k}.\ea \right.\ee
Due to \eqref{3.1} and \eqref{3.3}, it has $\n\in L^{\infty}(0,\infty;L^{1}),$ which implies that  $\n$ is bounded almost everywhere in  $\o\times (0,\infty)$.  Thus,  \bnn
 \lim_{k\rightarrow \infty} \n_{k}\rightarrow \n\quad  a.e. \quad \o\times (0,\infty).\enn  By   Lebesgue Dominated Convergence theorem, one has
\bnn \lim_{k\rightarrow \infty}\int \n_{k} =\int \n  =1.\enn
Therefore, for some  large number  $k$,
\be\la{3.7}\frac{3}{4} \le  \int \n_{k} \le \frac{5}{4}.\ee
It is clear from  \eqref{3.5} and  \eqref{3.6}  that
\be\la{3.8} \p   \n_{k}  = \left\{\ba
& \p \n,&{\rm in}\,\,\,\left\{(x,t):\, \frac{1}{k}\le \n\le k\right\}, \\
& 0,& {\rm in}\,\,\,\left\{(x,t):\, \n\le \frac{1}{k}\,\,{\rm or}\,\, \n\ge k\right\},\ea \right. \ee
where $\p=\p_{t}$ or $\p=\p_{x_{i}}.$
Moreover,
  by  \eqref{3.8} and \eqref{2.2},    it satisfies that
 \be\la{3.9}\n_{k}\in L^{\infty}(0,\infty; H^{1}) \cap  L^{2}(0,\infty; H^{1}) \ee  and that
  for all $p\in [2,\infty)$
\be\la{3.10}\ba  \left\|\n_{k}-\overline{\n_{k}}\right\|_{L^{p}}^{p}
&\le C\|\n_{k}\|_{L^{\infty}}^{p-2}\left\|\n_{k}-\overline{\n_{k}}\right\|_{L^{2}}^{2}\\
& \le C  \left\|\n_{k}-\overline{\n_{k}}\right\|_{L^{2}}^{2}\\
&\le C  \|\na \n_{k}\|_{L^{2}}^{2}\\
 &\le C \left\|\na \n \right\|_{L^{2}(\{x: \frac{1}{k}\le \n\le k\})}^{2}\\
  &\le C \left\|\na \n^{\frac{\g}{2}} \right\|_{L^{2}(\{x: \frac{1}{k}\le \n\le k\})}^{2}\\
&\le C \left\|\na \n^{\frac{\g}{2}}\right\|_{L^{2}}^{2},\ea\ee where $\overline{\n_{k}}$ is the average function of $\n_{k}$, and  the constant $C=C(p,k)$ depends on $p$ and $k.$
Integrating    \eqref{3.10} in time  and using   \eqref{3.5}, we get
 \be\la{3.11}\ba
  \int_{0}^{\infty}\left\| \n_{k} - \overline{\n_{k}} \right\|_{L^{p}}^{p} dt
  \le C(p,k). \ea\ee

\bigskip

We next prove that
   \be\la{3.12} \ba & \int_{0}^{\infty}   \left|\frac{d}{dt}\left\|\n_{k} - \overline{\n_{k}}\right\|_{L^{p}}^{p}\right| dt \le C(p,k).\ea\ee
For this purpose,   we  compute
 \be\la{3.13}\ba &\frac{d}{dt}\left\|\n_{k} - \overline{\n_{k}}\right\|_{L^{p}}^{p}\\
 &= p\frac{d}{dt} \overline{\n_{k}} \int
 \left|\n_{k}-\overline{\n_{k}}\right|^{p-2}\left(\overline{\n_{k}}-\n_{k}\right) +p \int \left|\n_{k}-\overline{\n_{k}}\right|^{p-2}\left(\n_{k}-\overline{\n_{k}}\right)  \p_{t}\n_{k}.\ea\ee
By virtue of  \eqref{3.9},  we are allowed to take in \eqref{1.9} the test function $\Phi=\n_{k}^{-\frac{1}{2}}\widetilde{\Phi}$ with any smooth function $\widetilde{\Phi}$  to deduce that
\be\la{3.14}\ba  &\int_{0}^{\infty}  \int  \n_{k}^{-\frac{1}{2}}\Lambda \widetilde{\Phi}\\
    & =-\int_{0}^{\infty}\int \sqrt{\n_{k}}\sqrt{\n_{k}}\u \div(\n_{k}^{-1} \widetilde{\Phi})-2\int_{0}^{\infty} \n_{k}^{-\frac{1}{2}}\u \na\sqrt{\n_{k}}\widetilde{\Phi}  \\
  & =-\int_{0}^{\infty}\int \u  \div\widetilde{\Phi}.\ea\ee
From  \eqref{3.14}  we see  that $\na\u $  is well  defined  in $ L^{2}$ in weak sense.   By uniqueness, it has
 \be\la{3.15}\na \u=\frac{\Lambda}{\sqrt{\n_{k}}}\in L^{2}(0,\infty; L^{2}).\ee

In view of  \eqref{1.5}, \eqref{3.8},  \eqref{3.15},  and the mass equation  $\eqref{n1}_{1}$, it satisfies
 \bnn\ba  \left|\frac{d}{dt} \overline{\n_{k}}\right|&=\left|\int   \p_{t}\n_{k}\right|\\
 &\le  \int_{\{x: \frac{1}{k}\le \n\le k\}} \left|\p_{t}\n \right|  \\
 &\le C\int_{\{x: \frac{1}{k}\le \n\le k\}} \left(|\u\cdot\na\n|+|\n_{k} \div \u|\right)\\
 &\le C\left(\|\sqrt{\n}\u\|_{L^{2}} \| \na \sqrt{\n}\|_{L^{2}(\{x: \frac{1}{k}\le \n\le k\})}+\|\n_{k}\|_{L^{1}}^{\frac{1}{2}}\|\sqrt{\n_{k}}\na \u \|_{L^{2}}\right) \\
&\le  C\left(\left\|\na \n^{\frac{\g}{2}} \right\|_{L^{2}(\{x: \frac{1}{k}\le \n\le k\})}+ \|\sqrt{\n_{k}} \na \u \|_{L^{2}}\right)\\
&\le  C\left(\left\|\na \n^{\frac{\g}{2}} \right\|_{L^{2}}+   \| \sqrt{\n_{k}}\na \u \|_{L^{2}}\right).
 \ea\enn
With the help of \eqref{3.6},  \eqref{3.10}, and  the  above inequality, we estimate the first integral on the right-hand side of \eqref{3.13} as
 \be\la{3.16}\ba &\left|p\frac{d}{dt} \overline{\n_{k}} \int
 \left|\n_{k}-\overline{\n_{k}}\right|^{p-2}\left(\overline{\n_{k}}-\n_{k}\right)\right|\\
  &\le C\left(\left\|\na \n^{\frac{\g}{2}}\right\|_{L^{2}}+   \| \sqrt{\n_{k}}\na \u \|_{L^{2}}\right)\int\left|
 \overline{\n_{k}}-\n_{k}\right| \\
  &\le C\left(\left\|\na \n^{\frac{\g}{2}}\right\|_{L^{2}}+   \| \sqrt{\n_{k}}\na \u \|_{L^{2}}\right)\left\|\overline{\n_{k}}-\n_{k}\right\|_{L^{2}} \\
  &\le   C\left(\left\|\na \n^{\frac{\g}{2}}\right\|_{L^{2}} +   \| \sqrt{\n_{k}}\na \u \|_{L^{2}} \right)\left\|\na \n^{\frac{\g}{2}}\right\|_{L^{2}}.\ea\ee

Next,  observe from \eqref{1.5},
   \eqref{3.15},  and Sobolev inequality that
 \be\la{3.17} \ba &\| \u\|_{L^{r}(\{x: \frac{1}{k}\le \n\})}\\
 &\le C\left(\|\u\|_{L^{2}(\{x: \frac{1}{k}\le \n\})}+\| \na \u\|_{L^{2}(\{x: \frac{1}{k}\le \n\})}\right)\\
 &\le C\left(\| \sqrt{\n}\u\|_{L^{2}(\{x: \frac{1}{k}\le \n\})}+\|\sqrt{\n_{k}} \na \u\|_{L^{2}}\right)\\
 &\le C\left(\| \sqrt{\n}\u\|_{L^{2}}+\|\sqrt{\n_{k}} \na \u\|_{L^{2}}\right)\\
 &\le C \left(1+\|\sqrt{\n_{k}}\na \u\|_{L^{2}}\right),\qquad \forall\quad r\in [2,6].\ea\ee
 Then,  by   \eqref{1.5},  \eqref{3.6}, \eqref{3.10}, \eqref{3.17},   and  $\eqref{n1}_{1}$, the last integral in  \eqref{3.13} satisfies
\be\la{3.18} \ba & \left|p \int \left|\n_{k}-\overline{\n_{k}}\right|^{p-2}\left(\n_{k}-\overline{\n_{k}}\right)  \p_{t}\n_{k}\right|\\
 &\le  C\int_{\{x: \frac{1}{k}\le \n\le k\}}   \left|\n_{k}-\overline{\n_{k}}\right|\left( |\n_{k} \div  \u|+|\u\na \n|\right) \\
 &\le C  \left\|\n_{k}-\overline{\n_{k}}\right\|_{L^{2}}\|\sqrt{\n_{k}}\na \u\|_{L^{2}}\\
 &\quad+C \left\|\n_{k}-\overline{\n_{k}}\right\|_{L^{3}} \|\na\n\|_{L^{2}(\{x:\frac{1}{k}\le \n\le k\})} \|\u\|_{L^{6}(\{x: \frac{1}{k}\le \n\le k\})} \\
 &\le  C \left\| \na \n^{\frac{\g}{2}}\right\|_{L^{2}}\|\sqrt{\n_{k}}\na \u\|_{L^{2}} \\
 &\quad+  C\| \na \sqrt{\n} \|_{L^{2}(\{x: \frac{1}{k}\le \n\le k\})}\|\na \n  \|_{L^{2}(\{x:\frac{1}{k}\le \n\le k\})}\left(1+\|\sqrt{\n_{k}}\na \u\|_{L^{2}}\right)\\
 &\le  C \left\| \na \n^{\frac{\g}{2}}\right\|_{L^{2}} \|\sqrt{\n_{k}}\na \u\|_{L^{2}}\\
 &\quad+C\| \na \sqrt{\n} \|_{L^{2}(\{x: \frac{1}{k}\le \n\le k\})}\|\na \n  \|_{L^{2}(\{x:\frac{1}{k}\le \n\le k\})}+  C \|\na \n  \|_{L^{2}(\{x:\frac{1}{k}\le \n\le k\})}\|\sqrt{\n_{k}}\na \u\|_{L^{2}}\\
 &\le C \left\| \na \n^{\frac{\g}{2}}\right\|_{L^{2}}^{2} +C\|\sqrt{\n_{k}}\na \u\|_{L^{2}}^{2}.\ea\ee

Therefore,   substituting   \eqref{3.16} and \eqref{3.18} back into \eqref{3.13},   integrating  the resultant  expression,   using \eqref{3.5} and \eqref{3.15}, we obtain
\bnn\ba &\int_{0}^{\infty}\left|\frac{d}{dt}\left\|\n_{k}-\overline{\n_{k}}\right\|_{L^{p}}^{p}\right|dt\\
 &\le C\int_{0}^{\infty}\left(\left\| \na \n^{\frac{\g}{2}}\right\|_{L^{2}}^{2} + \|\sqrt{\n_{k}}\na \u\|_{L^{2}}^{2}\right)dt\\
 &\le C.\ea\enn   This proves   \eqref{3.12}.

 \bigskip

It follows from  \eqref{3.11} and  \eqref{3.12} that
\be\la{3.19}\lim_{t\rightarrow \infty}\left\|\n_{k}-\overline{\n_{k}}\right\|_{L^{p}} =0,\quad \forall\,\,\,\, p\in [2,\infty).\ee
By  \eqref{3.6}, it has   $\n_{k}-\overline{\n_{k}}\in L^{\infty}(0,\infty;L^{\infty}).$ Thus,  sending $p\rightarrow \infty$ in \eqref{3.19} yields
  \bnn &\lim_{t\rightarrow \infty} \left\|\n_{k}-\overline{\n_{k}}\right\|_{L^{\infty}} \le \underline{\lim}_{p \rightarrow \infty}\lim_{t\rightarrow \infty}\left\|\n_{k}-\overline{\n_{k}}\right\|_{L^{p}}  =0,\enn
which implies, for some large   $T^{*}=T^{*}(k)<\infty$,
\be\la{3.20}  \ba& \left\| \n_{k}-\overline{\n_{k}}\right\|_{L^{\infty}}  \le \frac{1}{4},\qquad \forall\,\,\,t\in (T^{*},\infty).\ea\ee
By  \eqref{3.20} one deduces
\be\la{3.21}\ba \n_{k} \ge \overline{\n_{k}}-\frac{1}{4}
&=\int \n_{k} -\frac{1}{4}\\
&\ge  \frac{3}{4}-\frac{1}{4} =\frac{1}{2},\qquad a.e.\,\,\,\,(x,t)\in \o\times (T^{*},\infty),\ea\ee
where the last inequality owes to \eqref{3.7}.
For another hand,    from  \eqref{3.7} and    \eqref{3.20}  we have
\be\la{3.22}\ba  \n_{k}
 \le \overline{\n_{k} }+\frac{1}{4}
&= \int \n_{k}  +\frac{1}{4} \\
& \le \frac{5}{4} +\frac{1}{4} =  \frac{3}{2}, \quad a.e.\,\,\,\,(x,t)\in \o\times (T^{*},\infty).\ea\ee
The combination of  \eqref{3.21} with \eqref{3.22} shows that for some large but finite $k<\infty,$ there is a time point $T^{*}=T^{*}(k)<\infty$ such that
\be\la{3.23}\ba \frac{1}{2} \le  \n_{k} \le  \frac{3}{2}, \quad a.e.\,\,\,\,(x,t)\in \o\times (T^{*},\infty).\ea\ee
Utilizing  \eqref{3.6} once more,  we conclude \eqref{3.4} from \eqref{3.23}.
The proof is finished.

\bigskip

\subsection{Velocity regularity}
This subsection   aims  to improve the velocity regularity. Our  approach is using  an iteration technique.

\bigskip

For fixed $t_{0}>T^{*},$  we   define   the    cut-off function of the form
\be\la{3.24} \left\{\ba &0\le \xi(t)\in C^{1}(\mathbb{R}),\\
&\xi(t)\equiv0\,\,\,{\rm if}\,\,\,t\in (\infty,t_{0}-r'],\quad \xi(t)\equiv1\,\,\,{\rm if}\,\,\,t\in [t_{0}-r,\infty),\\
& |\xi'|\le \frac{2}{|r^{'}-r|},\quad
 \frac{t_{0}-T^{*}}{2}\le r<r'\le t_{0}-T^{*}.\ea\right. \ee

 If we test the momentum equations $\eqref{n1}_{2}$ against $\xi |\u|^{\beta}\u$ with $\beta\ge 0$, utilize $\eqref{n1}_{1}$ and  the Cauchy inequality, integrate over $\o\times (-\infty,2t_{0})$,  we infer
\be\la{3.25}\ba & \frac{1}{2+\beta}\sup_{t\in [t_{0}-r ,2t_{0}]}\int  \n|\u|^{2+\beta}\\
&\quad + \int_{t_{0}-r'}^{2t_{0}}\xi \int\left( \n|\u|^{\beta}|\na \u|^{2} +\beta  \n|\u|^{\beta}|\na |\u|  |^{2}\right) \\
&=\frac{1}{2+\beta} \int_{t_{0}-r'}^{2t_{0}}\p_{t}\xi\int\n|\u|^{2+\beta} +\int_{t_{0}-r'}^{2t_{0}}\xi \int   \n^{\g} {\rm div}\left(  |\u|^{\beta}\u\right)  \\
&\le  \frac{1}{2+\beta} \int_{t_{0}-r'}^{\tau}|\p_{t}\xi|\int \n|\u|^{2+\beta} \\
&\quad + \frac{1}{2}\int_{t_{0}-r'}^{2t_{0}}\xi \int\left( \n|\u|^{\beta}|\na \u|^{2} +\beta  \n|\u|^{\beta}|\na |\u|  |^{2}\right) +C\beta
 \int_{t_{0}-r'}^{2t_{0}}\int  \n^{2\g-1}|\u|^{\beta}.\ea\ee
Thanks to    \eqref{1.10} and \eqref{3.24},  the   \eqref{3.25}  satisfies
  \be\la{3.26}\ba & \sup_{t\in [t_{0}-r,2t_{0}]}\int  |\u|^{2+\beta} + (1+\beta)\int_{t_{0}-r}^{2t_{0}} \int \left( |\u|^{\beta}|\na \u|^{2} +  |\na |\u|^{\frac{2+\beta}{2}}|^{2}\right) \\
&\le    C \frac{(2+\beta)^{2}}{|r'-r|}
 \left(2+\int_{t_{0}-r'}^{2t_{0}}\int |\u|^{2+\beta}\right).\ea\ee

By  Sobolev inequality \eqref{2.3}, it takes
\bnn\ba &\left(\int_{t_{0}-r}^{2t_{0}}  \int |\u|^{\frac{5}{3}(2+\beta)} \right)^{\frac{3}{5}}\\
&\le C\left(\sup_{t\in [t_{0}-r,2t_{0}]}  \int |\u|^{(2+\beta)}+\int_{t_{0}-r}^{2t_{0}} \int \left|\na |\u|^{\frac{\beta+2}{2}}\right|^{2}+ \int_{t_{0}-r'}^{2t_{0}} \int|\u|^{\beta+2}\right) .\ea\enn
Making use of   \eqref{3.26}, we estimate the above inequality as
\be\la{3.27}\ba
  \left(\int_{t_{0}-r}^{2t_{0}}  \int |\u|^{\frac{5}{3}(2+\beta)} \right)^{\frac{3}{5}} \le  C \frac{(2+\beta)^{2}}{|r'-r|}
 \left(1+\int_{t_{0}-r'}^{2t_{0}}\int |\u|^{2+\beta}\right).\ea\ee

\bigskip

Next,  we apply  the  Moser-type iteration    to deduce   the desired \eqref{1.11}.

Select
  \bnn 2+\beta=\left(\frac{5}{3}\right)^{k},\quad  r^{'}=r_{k}, \quad   r=r_{k+1}=\frac{(t_{0}-T^{*})}{2}\left(1+\frac{1}{2^{k+1}}\right).\enn
From \eqref{3.27} we receive that, for some  constant $C=C(t_{0},\,T^{*}),$
 \be\la{3.28}\ba&
  \left(\int_{t_{0}-r_{k+1}}^{2t_{0}}  \int |\u|^{\left(\frac{5}{3}\right)^{k+1}} \right)^{\left(\frac{3}{5}\right)^{k+1}}\\
   &\le  C \left( 2\cdot\frac{5}{3}\right)^{2k\cdot \left(\frac{3}{5}\right)^{k}}
 \left[1+\int_{t_{0}-r_{k}}^{2t_{0}}\int |\u|^{\left(\frac{5}{3}\right)^{k}}\right]^{\left(\frac{3}{5}\right)^{k}}\\
 &\le  C3^{2k\cdot \left(\frac{3}{5}\right)^{k}} \cdot 2^{   \left(\frac{3}{5}\right)^{k}}
 \left[1+\left(\int_{t_{0}-r_{k}}^{2t_{0}}\int |\u|^{\left(\frac{5}{3}\right)^{k}}\right)^{\left(\frac{3}{5}\right)^{k}}\right]\\
 &\le  C  3^{3k\cdot \left(\frac{3}{5}\right)^{k}}
 \left[1+\left(\int_{t_{0}-r_{k}}^{2t_{0}}\int |\u|^{\left(\frac{5}{3}\right)^{k}}\right)^{\left(\frac{3}{5}\right)^{k}}\right],  \ea\ee
where  the constant $C$ is independent of  $ k.$

We proceed \eqref{3.28} in two cases:

 {\it Case 1.} If there exists an infinite  sequence  to satisfy
\be\la{3.29} 1\ge \left(\int_{t_{0}-r_{k}}^{2t_{0}}\int |\u|^{\left(\frac{5}{3}\right)^{k}}\right)^{\left(\frac{3}{5}\right)^{k}}.\ee
 Then \be\la{3.30}\ba &\|\u\|_{L^{\infty}\left(\frac{(t_{0}+T^{*})}{2},2t_{0};L^{\infty} \right)}\\
  &=\lim_{k\rightarrow \infty}\left(\int_{\frac{(t_{0}+T^{*})}{2}}^{2t_{0}}\int |\u|^{\left(\frac{5}{3}\right)^{k}}\right)^{\left(\frac{3}{5}\right)^{k}}\\
  &\le \lim_{k\rightarrow \infty}\left(\int_{t_{0}-r_{k}}^{2_{0}}\int |\u|^{\left(\frac{5}{3}\right)^{k}}\right)^{\left(\frac{3}{5}\right)^{k}}\le  1.\ea\ee

 {\it Case 2.} While if  \eqref{3.29} is false, there should be a finite number  $k_{0}\ge 1$ such that  \eqref{3.28} satisfies, for all $k\ge k_{0},$
\be\la{3.31}\ba
\left(\int_{t_{0}-r_{k+1}}^{2t_{0}}  \int |\u|^{\left(\frac{5}{3}\right)^{k+1}} \right)^{\left(\frac{3}{5}\right)^{k+1}} \le  C3^{3k\cdot \left(\frac{3}{5}\right)^{k}}
 \left(\int_{t_{0}-r_{k}}^{2t_{0}}\int |\u|^{\left(\frac{5}{3}\right)^{k}}\right)^{\left(\frac{3}{5}\right)^{k}}.\ea\ee
By  deduction, it gives  from  \eqref{3.31} that
  \be\la{3.32}\ba \left(\int_{t_{0}-r_{k+1}}^{2t_{0}}  \int |\u|^{\left(\frac{5}{3}\right)^{k+1}} \right)^{\left(\frac{3}{5}\right)^{k+1}}
  &\le   C^{a}3^{b} \left(\int_{t_{0}-r_{k}}^{2t_{0}}\int |\u|^{\left(\frac{5}{3}\right)^{k}} \right)^{\left(\frac{3}{5}\right)^{k}},\quad \forall\,\,\,k\ge k_{0},\ea\ee
  with   \bnn a\le \sum_{k=1}^{\infty} \left(\frac{3}{5}\right)^{k}<\infty,\quad b\le 3\sum_{k=1}^{\infty}k\cdot \left(\frac{3}{5}\right)^{k}<\infty.\enn
Taking limit  $k\rightarrow\infty$   in \eqref{3.32} yields
\be\la{3.33}\ba  &\|\u\|_{L^{\infty}\left(\frac{(t_{0}+T^{*})}{2},2t_{0};L^{\infty}\right)}\\&\le C\lim_{k\rightarrow \infty}\left(\int_{t_{0}-r_{k}}^{2t_{0}}\int |\u|^{\left(\frac{5}{3}\right)^{k}} \right)^{\left(\frac{3}{5}\right)^{k}}\\
&\le C\left(\int_{t_{0}-r_{k_{0}}}^{2t_{0}}\int |\u|^{\left(\frac{5}{3}\right)^{k_{0}}} \right)^{\left(\frac{3}{5}\right)^{k_{0}}}\\
&\le C(k_{0})\left(1+\int_{\frac{t_{0}+T^{*}}{2}}^{2t_{0}} \int |\u|^{2+\beta}\right),\ea\ee
where  the last inequality is  from a   finite many time iteration  in  \eqref{3.28}.

Therefore, the  combination of  \eqref{3.30} with   \eqref{3.33} generates
 \be\la{3.34}\ba \|\u\|_{L^{\infty}\left(\frac{(t_{0}+T^{*})}{2},2t_{0};L^{\infty} \right)}& \le  C+C \int_{\frac{t_{0}+T^{*}}{2}}^{2t_{0}} \int |\u|^{2+\beta}.\ea\ee
In terms of   \eqref{1.5}, \eqref{1.10}, and Sobolev inequality \eqref{2.1}, we choose $\beta\in [0,4]$ in \eqref{3.34}
 and obtain
 \be\la{3.35}\ba &\|\u\|_{L^{\infty}\left(\frac{(t_{0}+T^{*})}{2},2t_{0};L^{\infty} \right)}\\& \le  C+C \int_{\frac{t_{0}+T^{*}}{2}}^{2t_{0}} \int |\u|^{2+\beta}\\
 &\le  C+C t_{0} \sup_{t\ge t_{0}} \int |\u|^{2}+C \int_{\frac{t_{0}+T^{*}}{2}}^{2t_{0}} \int |\na \u|^{2}\\
 &\le  C+C t_{0} \sup_{t\ge 0} \int \n|\u|^{2}+C \int_{0}^{\infty} \int \n|\na \u|^{2}\\
 &\le C.\ea\ee

Finally, repeating  the above argument over the intervals $[n t_{0},((n+1)t_{0})]$ for $n=2,3...$, we obtain \eqref{1.11},  the required.

\bigskip

\subsection{Energy conservation} As a result of   \eqref{1.10} and  \eqref{1.11},
we  show  that    the weak solution $(\n,\u)$    conserves its energy equality for all $t>T^{*}$.

\bigskip

First, owing to \eqref{1.10}, the exactly same deduction of    \eqref{3.15} yields  for some $t_{1}>T^{*}$
\be\la{3.36}\Lambda =\sqrt{\n} \na\u\in L^{2}\left(t_{1},\infty;L^{2}\right). \ee
This proves the latter part in \eqref{1.11}. Moreover, by    \eqref{1.6}, it satisfies
\be\la{3.37}  \int \mathcal{E}(x,t)dx+  \int_{t_{1}}^{t}\int  \n|\na\u|^{2} \le \int \mathcal{E}(x,t)dx,\quad \forall\,\,\,\, t\ge t_{1}. \ee
To be continued, we present the following proposition
  \begin{prop}\la{p} For  almost all  $t_{1}>T^{*}$ and for all $t \in [t_{1},\infty)$,  we have  equality sign in \eqref{3.37}, that is to say, the desired \eqref{1.12}  is true, provided that
  \be\la{3.38}\ba & 0<\underline{C} \le \n(x,t)\le \overline{C}<\infty, \quad {\rm in}\quad \o\times (t_{1},\infty),\\
&\qquad\qquad \u\in L^{4}\left(t_{1},t ;L^{6}\right),\\
&\quad \|\u(\cdot,t_{1})\|_{L^{q_{0}}}\le  C\quad {\rm for\,\,some}\quad q_{0}>3.\ea\ee
\end{prop}
\begin{proof} The proof of   Proposition \ref{p} is available   \cite[Theorem 1.1]{yu2}.\end{proof}

\bigskip

In this connection, to  prove \eqref{1.12}  it suffices to  verify   \eqref{3.38}.  In fact, the first two conditions $\eqref{3.38}_{1} $ and $\eqref{3.38}_{2}$ are from  \eqref{1.10} and  \eqref{1.11} respectively.   Using   \eqref{1.11} once more,  we apply  the  Lebesgue Point theorem and obtain that,  for a.e.  $t_{1}\in (T^{*},\infty),$
\bnn \|\u(\cdot,t_{1})\|_{L^{q_{0}}}^{q_{0}}=\lim_{\ep\rightarrow0}\int_{t_{1}}^{t_{1}+\ep}\int |\u(x,s)|^{q_{0}}dxds\le C.\enn
This confirms $\eqref{3.38}_{3}$.

\bigskip

\subsection{Exponential asymptotics} We  shall    prove that the $L^{2}$ norm of the  weak solution $(\n,\u)$  decays exponentially to its equilibrium as time goes to infinity.  The following    operations are assumed to be  carried  out over   $(T^{*},\infty).$

\bigskip

Denote  the material derivative of $f$ by  $\frac{d}{dt}f=\p_{t}+\u\cdot \na f$.
Owing to   \eqref{1.14},  $\eqref{n1}_{1}$, and the transport theorem, we  compute
\be\la{3.40}\ba \int \na \n^{\g} (\u-(\u)_{a})&=-\int (\n^{\g}-\overline{\n}^{\g})\div \u\\
&=\int \n\frac{(\n^{\g}-\overline{\n}^{\g})}{\n^{2}} \frac{d}{dt}\n \\
&=\int \n \frac{d}{dt}\int_{\overline{\n}}^{\n}\frac{s^{\g}-\overline{\n}^{\g}}{s^{2}}ds\\
&=\frac{d}{dt}\int \left(\n\int_{\overline{\n}}^{\n}\frac{s^{\g}-\overline{\n}^{\g}}{s^{2}}\right)ds.\ea\ee
From   $\eqref{n1}_{2}$ one  has
\be\la{3.41} \n (\u-(\u)_{a})_{t}+\n \u\cdot \na (\u-(\u)_{a})+\na \n^{\g}=\div (\n\na \u).\ee
Multiplying  \eqref{3.41}  by $ \u-(\u)_{a}$ and using \eqref{3.40}, we  obtain
\be\ba\la{3.42}\frac{d}{dt}\int \n\left(\frac{1}{2} |\u-(\u)_{a}|^{2}+\int_{\overline{\n}}^{\n}\frac{s^{\g}-\overline{\n}^{\g}}{s^{2}}ds\right)+\int \n |\na \u|^{2}=0.\ea\ee

Next, for  the operator $\mathcal{B}$   defined in Lemma \ref{lem2.4},  we have
 \bnn\ba &\int \n (\u-(\u)_{a})_{t}  \mathcal{B}(\n-\overline{\n})\\
 &=\frac{d}{dt} \int \n (\u-(\u)_{a})  \mathcal{B}(\n-\overline{\n})\\
 &\quad - \int \n_{t} (\u-(\u)_{a})  \mathcal{B}(\n-\overline{\n})+\int \n (\u-(\u)_{a}) \mathcal{B}\div (\n\u) \\
 &=\frac{d}{dt} \int \n (\u-(\u)_{a})\mathcal{B}(\n-\overline{\n})\\
 &\quad - \int \n_{t} (\u-(\u)_{a})  \mathcal{B}(\n-\overline{\n})+\int \n\left(\n |\u-(\u)_{a}|^{2}+ (\u-(\u)_{a})(\n-\overline{\n})(\u)_{a}\right) \ea \enn
 and  \bnn \ba &\int \n\u \na (\u-(\u)_{a})  \mathcal{B}(\n-\overline{\n})\\
 &= - \int \div(\n\u) (\u-(\u)_{a})  \mathcal{B}(\n-\overline{\n})-\int \n\u   (\u-(\u)_{a})  \na  \mathcal{B}(\n-\overline{\n}).\ea \enn
Taking  the last two inequalities and  $\eqref{n1}_{1}$ into account,  we multiply   $\eqref{3.41}$  against  $ \mathcal{B}(\n-\overline{\n})$ and compute
\be\ba\la{3.43}  & \frac{d}{dt} \int \n (\u-(\u)_{a})  \mathcal{B}(\n-\overline{\n})\\
&+\int \n^{2}   |\u-(\u)_{a}|^{2}+\int \n(\n-\overline{\n})  (\u-(\u)_{a})(\u)_{a}\\
& -\int \n\u (\u-(\u)_{a}) \na  \mathcal{B}(\n-\overline{\n})-\int (\n^{\g}-\overline{\n}^{\g})(\n-\overline{\n})\\
&+\int \n\na \u \na  \mathcal{B}(\n-\overline{\n}) =0.\ea\ee
For   small constant $\de\in (0,1)$ to be determined,  we multiply  \eqref{3.43} by $\de,$  add  it up to \eqref{3.42}, to receive
\be\la{3.44} \frac{d}{dt} A +B=0,\ee where
 \bnn\ba A:= \int\n\left(\frac{1}{2} |\u-(\u)_{a}|^{2}+\int_{\overline{\n}}^{\n}\frac{s^{\g}
 -\overline{\n}^{\g}}{s^{2}}ds\right)-\de\int \n(\u-(\u)_{a}) \mathcal{B}(\n-\overline{\n})
 \ea\enn
 and
 \bnn \ba B&=\int \n |\na \u|^{2}+\de\int (\n^{\g}-\overline{\n}^{\g})(\n-\overline{\n})\\
&\quad-\de\int \n^{2}   |\u-(\u)_{a}|^{2}-\de\int \n(\n-\overline{\n})(\u-(\u)_{a})(\u)_{a} \\
 & \quad+\de\int \n\u   (\u-(\u)_{a})\na\mathcal{B}(\n-\overline{\n}) -\de\int \n\na \u \na  \mathcal{B}(\n-\overline{\n}).\ea\enn
We  estimate $A$ and $B$ in below.

First, due to  \eqref{1.10},  Lemma \ref{lem2.4}, the  Poincar\'{e} inequality, it  satisfies   that
 \bnn C^{-1}|\n-\overline{\n}|^{2}\le \n \int_{\overline{\n}}^{\n}\frac{s^{\g}-\overline{\n}^{\g}}{s^{2}}ds \le C|\n-\overline{\n}|^{2}\enn
and
\bnn\ba&  \left|\int \n(\u-(\u)_{a})\cdot \mathcal{B}(\n-\overline{\n})\right|\\
&\le  C\|\u-(\u)_{a}\|_{L^{2}}^{2}+C\left\| \mathcal{B}(\n-\overline{\n})\right\|_{L^{2}}^{2} \\
 &\le   C\|\u-(\u)_{a}\|_{L^{2}}^{2}+C\left\|\na\mathcal{B}(\n-\overline{\n})\right\|_{L^{2}}^{2}\\
 &\le   C\|\u-(\u)_{a}\|_{L^{2}}^{2}+C\left\|\n-\overline{\n}\right\|_{L^{2}}^{2}.\ea\enn
  The last two inequalities together with \eqref{1.10} guarantee,   for a   small $\de,$  there exists some constant $C_{1}=C_{1}(\de)>1$ so that
 \be\la{3.45}C_{1}^{-1}\left(\|\u-(\u)_{a}\|_{L^{2}}^{2}+\|\n-\overline{\n}\|_{L^{2}}^{2}\right)\le  A\le C_{1}\left(\|\u-(\u)_{a}\|_{L^{2}}^{2}+\|\n-\overline{\n}\|_{L^{2}}^{2}\right).\ee

\bigskip

Next to estimate $B$.  Using
  \eqref{1.10}, Lemma \ref{lem2.4}, and  the Poincar\'{e} inequality, one deduces
\be\la{3.46} \ba &\qquad \int (\n^{\g}-\overline{\n}^{\g})(\n-\overline{\n})\ge C^{-1}\| \n-\overline{\n}\|_{L^{2}}^{2},\\
 &\|\mathcal{B}(\n-\overline{\n})\|_{L^{2}}\le C\|\na\mathcal{B}(\n-\overline{\n})\|_{L^{2}}\le C \|\n-\overline{\n}\|_{L^{2}}  \le C\|\na \n\|_{L^{2}},\\
 & \qquad \qquad  \|\u-(\u)_{a}\|_{L^{2}}\le C\|\na \u\|_{L^{2}}.\ea \ee
 \begin{remark} The validity of the  last inequality  $\eqref{3.46}_{3}$ can be understand as follows:
It gives from  \eqref{3.1} and \eqref{3.2} that  $(\u)_{a}=\frac{\int \m_{0}}{\int \n_{0}}=\frac{\int \n\u}{\int \n}.$
Let us mollify $\u$ by $\u^{\ep}=\u*\zeta(x)$ with $\zeta$ the  Friedrichs mollifier, then, by mean value theorem,  we have  $\u^{\ep}(\xi)=(\u^{\ep})_{a}=\frac{\int \n \u^{\ep}}{\int \n}$ for  some $\xi\in \o$.  Hence
 \bnn\ba &\|\u-(\u)_{a}\|_{L^{2}}\\&\le \|\u-\u^{\ep}\|_{L^{2}}+\|\u^{\ep}-(\u^{\ep})_{a}\|_{L^{2}}+\|(\u^{\ep})_{a}-(\u)_{a}\|_{L^{2}}\\
 &\le C\|\u-\u^{\ep}\|_{L^{2}}+C\|\na\u^{\ep}\|_{L^{2}}.\ea\enn
By  \eqref{1.10}, sending $\ep\rightarrow0$ in above inequality gives $\eqref{3.46}_{3}.$
 \end{remark}
In accordance with   \eqref{3.46},   \eqref{1.10},  \eqref{1.11},    we deduce
 \be\la{3.47} \ba B& \ge C\|\na \u\|_{L^{2}}^{2}+C\de  \| \n-\overline{\n} \|_{L^{2}}^{2}\\
 &\quad - C\de \|\na \u\|_{L^{2}}^{2}-C\de  \| \na\u  \|_{L^{2}} \|\n-\overline{\n}\|_{L^{2}}\\
 & \ge C\|\na \u\|_{L^{2}}^{2}+C\de  \| \n-\overline{\n} \|_{L^{2}}^{2}\\
 &\quad -  C\de \|\na\u   \|_{L^{2}}^{2}-C\de\left(\de^{-\frac{1}{2}}\|\na \u\|_{L^{2}}^{2}+ \de^{\frac{1}{2}}\| \n-\overline{\n}\|_{L^{2}}^{2}\right)\\
 &\ge C(1-\de^{\frac{1}{2}})\|\na \u\|_{L^{2}}^{2}+C\de (1-\de^{\frac{1}{2}}) \| \n-\overline{\n} \|_{L^{2}}^{2}\\
 &\ge C_{2} \left(\|\na \u\|_{L^{2}}^{2}+  \| \n-\overline{\n} \|_{L^{2}}^{2}\right),\ea\ee where the constant $C_{2}=C_{2}(\de)>0$, as long as $\de\in (0,1)$ is taken small enough.

 Therefore, with  \eqref{3.45} and \eqref{3.47} in hand, we may choose  $\sigma C_{1}\le \frac{1}{2}C_{2}$ and multiply \eqref{3.44} by $e^{\sigma t}$, to deduce
 \bnn \ba & \frac{d}{dt} \left(e^{\sigma t} \left(\|\u-(\u)_{a}\|_{L^{2}}^{2}+  \| \n-\overline{\n}\|_{L^{2}}^{2}\right)\right)\\
 &\quad +  C_{3} e^{\sigma t} \left(\|\na\u\|_{L^{2}}^{2}+  \| \n-\overline{\n}\|_{L^{2}}^{2}\right) \le 0  \ea\enn for some $C_{3}=\frac{1}{2}C_{1}^{-1}C_{2}.$ \quad
Integration of  the above inequality  yields
  \be\la{3.48}\ba &  e^{\sigma t} \left(\|\u-(\u)_{a}\|_{L^{2}}^{2}+  \| \n-\overline{\n}\|_{L^{2}}^{2}\right)+\int_{t_{1}}^{\infty}e^{\sigma t} \|\na\u\|_{L^{2}}^{2}dt \\
 &\le C e^{\sigma t_{1}} \left(\|\na\u\|_{L^{2}}^{2}+  \| \n-\overline{\n}\|_{L^{2}}^{2}\right)(t_{1})\\
 &\le C,   \ea\ee for some   $t_{1}>T^{*}$ and for any $t\ge t_{1}.$
 This yields  the   desired \eqref{1.13}.
Thus, the   proof of Theorem \ref{t1} is completed.

\section{Proof of Theorem \ref{t2}}

 \subsection{Exponential decay of $\|\na\n\|_{L^{2}}$} In case of zero initial momentum assumption \eqref{1.15}, we show that the $L^{2}$-norm of the derivative of density decays exponentially  to zero.

 \bigskip

Thanks to   \eqref{1.15} and  \eqref{3.48},
\be\la{4.1}\ba &  e^{\sigma t}  \|\u\|_{L^{2}}^{2}+\int_{t_{1}}^{\infty}e^{\sigma t} \|\na\u\|_{L^{2}}^{2}dt \le C,   \quad  t> t_{1}.\ea\ee
Recalling   \eqref{1.5} and   \eqref{1.10}, for some small $\de\in (0,1)$ one  has
\bnn\ba   \de \|\na\n\|_{L^{2}}^{2}-C(\de)\|\u\|_{L^{2}}^{2}\le \int \n|\u+\na\ln\n|^{2}& \le  \|\na\n\|_{L^{2}}^{2}+C \|\u\|_{L^{2}}^{2}\ea\enn
   and  \bnn\ba \int_{\o}\left(|\na \n^{\frac{\g}{2}}|^{2}+\Lambda^{2}\right)\ge \int_{\o} |\na \n^{\frac{\g}{2}}|^{2}\ge  C_{4}\|\na\n\|_{L^{2}}^{2}.\ea\enn
 With   the last two inequalities, we multiply \eqref{1.7} by $e^{\sigma_{1}t}$ and deduce
 \be\la{4.2}\ba & \de e^{\sigma_{1}t} \|\na\n\|_{L^{2}}^{2} +C_{4}\int_{t_{1}}^{\infty}e^{\sigma_{1} t} \|\na\n\|_{L^{2}}^{2}dt\\
 &\le C(\de) e^{\sigma_{1}t}\|\u\|_{L^{2}}^{2}+\sigma_{1}\int_{t_{1}}^{\infty} e^{\sigma_{1}t}\left( \|\na\n\|_{L^{2}}^{2}+C \|\u\|_{L^{2}}^{2}\right).\ea\ee
Select  $0<\sigma_{1}< \min\{\sigma,\,C_{4}\}$ in \eqref{4.2},  and utilize  \eqref{4.1}, to discover
 \bnn\ba   e^{\sigma_{1}t} \|\na\n\|_{L^{2}}^{2} + \int_{t_{1}}^{\infty}e^{\sigma_{1} t} \|\na\n\|_{L^{2}}^{2}dt\le C(\de),\ea\enn
 which yields the  required \eqref{1.16}.

 \bigskip

 \subsection{Higher  regularity    in dimension two}
   In the rest of this paper, we show that $(\n,\u)$ becomes a   strong solution to \eqref{n1}  in two  dimensional case when time is large.

   \bigskip

 In terms of  \eqref{1.5}, \eqref{1.10}, \eqref{1.11}, and the mass equation $\eqref{n1}_{1}$, it satisfies that   for all $t>T^{*}$
\be\la{4.3}\ba &\n \in L^{\infty}(t,\infty; L^{\infty})\cap L^{2}(t,\infty; H^{1}) \cap H^{1}(t,\infty; L^{2}),\\
&\qquad\u\in L^{\infty}(t,\infty;   L^{\infty})\cap L^{2}(t,\infty; H^{1}).\ea\ee
Using \eqref{4.3}, we   multiply   $\eqref{n1}_{2}$ by $\u_{t}$ and integrate  it over $(t,\infty)$, to find
\be\la{4.4}\ba &\frac{1}{2}\frac{d}{dt}\int \n |\na\u|^{2}+ \int \n |\u_{t}|^{2}+\\
&= - \int \n\u\cdot \na \u\cdot \u_{t}- \int \u_{t}\cdot \na\n^{\g}+\frac{1}{2} \int \n_{t} |\na\u|^{2}\\
&\le \frac{1}{4} \int \n |\u_{t}|^{2} + C\int\left(|\na\n|^{2}+|\na\u|^{2}\right)+\frac{1}{2} \int \n_{t} |\na\u|^{2}. \ea\ee
By   Lemma \ref{lem2.3} and the momentum equations, it satisfies
\be\la{4.5}\ba &\left\|\na^{2}\u\right\|_{L^{\frac{3}{2}}}\\
&\le C\left\|\lap\u\right\|_{L^{\frac{3}{2}}}\\
&\le C\left(\|\div(\n\na\u)\|_{L^{\frac{3}{2}}}+\|\na\n\na\u\|_{L^{\frac{3}{2}}}\right)\\
&\le C\left(\|\sqrt{\n}\u_{t}\|_{L^{2}}+\|\na\u \|_{L^{2}}+\|\na\n \|_{L^{2}}\right)+C_{4}\|\na\n\|_{L^{2}}\|\na\u\|_{L^{6}}.\ea\ee
Remember that the   Sobolev inequality in dimension two, it has
 \bnn \|\na\u\|_{L^{6}} \le   C  \left(\|\na^{2}\u\|_{L^{\frac{3}{2}}}+\|\na\u\|_{L^{2}}\right),\enn
which combining with  \eqref{1.16} shows, for some large  $\widetilde{t}_{1}\ge t_{1}$, \bnn \ba& C_{4}\|\na\n\|_{L^{2}}\|\na\u\|_{L^{6}}\\
&\le C_{4}e^{-\sigma_{1} t}\left(\|\na^{2}\u\|_{L^{\frac{3}{2}}}+\|\na\u\|_{L^{2}}\right)\\
&\le \frac{1}{2}\left(\|\na^{2}\u\|_{L^{\frac{3}{2}}}+\|\na\u\|_{L^{2}}\right),\quad \forall\,\,\,\,t>\widetilde{t}_{1}.\ea\enn
Insert the last  inequality  into \eqref{4.5} brings us to
\be\la{4.6}\ba \left\|\na^{2}\u\right\|_{L^{\frac{3}{2}}}
 \le C\left(\|\sqrt{\n}\u_{t}\|_{L^{2}}+\|\na\u \|_{L^{2}}+\|\na\n \|_{L^{2}} \right).\ea\ee
With the help of  \eqref{4.3}  and \eqref{4.6}, we estimate  the last integral in \eqref{4.4} as
\be\la{4.7}\ba  \int \n_{t} |\na\u|^{2}&= \int \n\u \cdot \na |\na\u|^{2}\\
&\le C  \left\|\n\u\right\|_{L^{\infty}}\left\|\na^{2}\u\right\|_{L^{\frac{3}{2}}}\|\na\u\|_{L^{3}} \\
&\le  \de \left\|\na^{2}\u\right\|_{L^{\frac{3}{2}}}^{2}+ C(\de)   \|\na\u\|_{L^{2}}^{2}\\
&\le C\de \left(\|\sqrt{\n}\u_{t}\|_{L^{2}}+\|\na\n \|_{L^{2}} \right)^{2} +C(\de) \|\na\u \|_{L^{2}}^{2}.\ea\ee
If we  substitute   \eqref{4.7} into \eqref{4.4}, and choose $\de$ suitably small,   we get
\be\la{4.8}\ba \frac{d}{dt}\int \n |\na\u|^{2}+ \int \n |\u_{t}|^{2}\le C\int\left(|\na\n|^{2}+|\na\u|^{2}\right). \ea\ee
Hence,   integrating    \eqref{4.8}  after multiplied by $e^{\sigma_{1}t}$, utilizing  \eqref{4.1}, \eqref{1.10} \eqref{1.16},  we obtain
\be\la{4.9}\ba  e^{\sigma_{1}t}\|\na\u(\cdot,t)\|_{L^{2}}+ \int_{\widetilde{t}_{1}}^{t}e^{\sigma_{1}s}\|\p_{s}\u\|_{L^{2}}^{2}ds\le C,\quad t>\widetilde{t}_{1}. \ea\ee
The combination of \eqref{4.9} with  \eqref{1.16} generates \eqref{1.18}.

As a  result  of   \eqref{4.6} and \eqref{4.9},   we have
\bnn\la{j12}\u\in L^{\infty}(t,\infty; H^{1})\cap L^{2}(t,\infty; W^{2,\frac{3}{2}}) \cap H^{1}(t,\infty; L^{2}),\quad t>\widetilde{t}_{1}.\enn
This proves  the desired \eqref{1.17}.

The proof of Theorem \ref{t2} is thus completed.



\bigskip

\begin {thebibliography} {99}


\bibitem{adams} R. Adams,    Sobolev spaces, New York: Academic Press  (1975).

\bibitem{ag} S. Agmon, A.  Douglis, L. Nirenberg, {Estimates near the boundary for solutions
of elliptic partial differential equations satisfying general boundary conditions,} I,
Comm. Pure Appl. Math., {\bf12} (1959),623-727; II, Comm. Pure Appl. Math. {\bf17}
(1964), 35-92.

\bibitem{anton} S. N. Antontsev; A. V.  Kazhikhov;  V. N.  Monakhov, Boundary Value Problems
in Mechanics of Nonhomogeneous Fluids. Amsterdam, New York: North-Holland, 1990.

\bibitem{kim1} Y. Cho,  H. Choe, H.  Kim, {\em Unique solvability of the initial boundary value problems for compressible viscous fluids,} J. Math. Pures Appl., {\bf 83} (9)  (2004),  243-275.

\bibitem{des} B. Desjardins, Regularity of weak solutions of the compressible isentropic Navier Stokes equations,
 Commun. Partial  Differ. Equ.,  \textbf{22} (1997), 977-1008.


\bibitem{bres}D. Bresch;  B. Desjardins, { Some diffusive capillary models of Korteweg type,}  C. R.
Math. Acad. Sci. Paris, Section Mecanique, \textbf{332}(11), (2004), 881-886.
\bibitem{bres1}  D. Bresch;  B. Desjardins, {Existence of global weak solution for 2D viscous shallow
water equations and convergence to the quasi-geostrophic model,} Comm. Math. Phys.,
\textbf{238}(1-2)  (2003),  211-223.

\bibitem{liang} R. Chen; Z. Liang; D. Wang and  R. Xu, Energy equality in compressible fluids with physical boundaries, SIAM J. Math. Anal., \textbf{52} (2020),  1363-1385.

 \bibitem{di}    R.  DiPerna; P.  Lions, Ordinary differential equations, transport theory and
 Sobolev spaces, Invent. Math. \textbf{98} (1989), 511-547.


\bibitem{fei} E. Feireisl; A. Novotn\'{y}; H. Petzeltov\'{a}, On the existence of globally defined weak solutions to the Navier-Stokes equations, J. Math. Fluid Mech., \textbf{3}  (2001),  358-392.



\bibitem{kim} H. Choe; B. Jin, { Regularity of weak solutions of the compressible Navier-Stokes equations} J. Korean  Math. Soc.,  \textbf{40} (2003), 1031-1050.

\bibitem{g}   P. Gent, The energetically consistent shallow water equations, J. Atmos. Sci., \textbf{50} (1993),  1323-1325.

    \bibitem{guo}  Z. Guo;  Q. Jiu;  Z.  Xin, { Spherically symmetric isentropic compressible flows with density-dependent viscosity coefficients,}  SIAM J. Math. Anal., \textbf{39}  (2008), 1402-1427.

\bibitem{has} B. Haspot, Existence of global strong solution for the compressible Navier-Stokes equations with degenerate viscosity coefficients in 1D, Mathematische Nachrichten,  \textbf{291}(14-15)  (2018),  2188-2203.

\bibitem{hoff} D. Hoff,   Global existence for 1D, compressible, isentropic Navier-Stokes equations
with large initial data, Trans. Amer. Math. Soc., \textbf{303} (1987),  169-181.

\bibitem{hoff1} D. Hoff, Global solutions of the Navier-Stokes equations for multidimensional, compressible flow with
discontinuous initial data. J. Differ. Equ. \textbf{120} (1995), 215-254

 \bibitem{gp} J.  Gerbeau; B.  Perthame,  Derivation of viscous Saint-Venant system for laminar shallow water, numerical validation, Discrete Contin. Dyn. Syst. Ser B, \textbf{1}
(2001),  89-102.

\bibitem{jiang1} S. Jiang, Large-time behavior of solutions to the equations of a one-dimensional
viscous polytropic ideal gas in unbounded domains, Comm. Math. Phys., \textbf{200}
(1999), 181-193.

\bibitem{jiang2} S. Jiang,  Remarks on the asymptotic behaviour of solutions to the compressible
Navier-Stokes equations in the half-line, P. Roy. Soc. Edinb. A,  \textbf{132}
(2002), 627-638.

\bibitem{jiang3} S. Jiang; Z. Xin; P. Zhang,   Golobal weak solutions to 1D compressible isentropic
Navier-Stokes equations with density-dependent viscosity, Methods Appl. Anal.,
\textbf{12}  (2005),  239-252.


\bibitem{jiangzhang}S. Jiang; P. Zhang, { On spherically symmetric solutions of the compressible isentropic Navier-Stokes
equations,}  Comm. Math. Phys., \textbf{215}  (2001), 559-581.


\bibitem{jiu} Q. Jiu;  Z. Xin,  { The Cauchy problem for 1D compressible flows with density-dependent viscosity coefficients,} Kinet. Relat. Mod., \textbf{1}(2)  (2008), 313-330.

\bibitem{kv}  M.  Kang;  A. Vasseur, Global smooth solutions for 1D barotropic Navier-Stokes equations with a large class of degenerate viscosities,
J. Nonlinear Sci., \textbf{30}(4)  (2020),  1703-1721.

\bibitem{lad} O. Ladyzenskaja; V. Solonnikov; N. Uraltseva,   Linear and quasi-linear
equations of parabolic type. Translated from the Russian by S. Smith. Translations of Mathematical Monographs, Vol. 23. American Mathematical Society,
Providence, R.I., 1968.

\bibitem{lau} L. Laudau; E. Lifshitz, { Electrodynamics of Continuous Media}, 2nd edn. Pergamon, New York (1984)

\bibitem{llx} H. Li;  J. Li;  Z. Xin, { Vanishing of vacuum states and blow-up phenomena of the compressible Navier-Stokes equations,}
 Comm. Math. Phys., \textbf{281}(2)  (2008), 401-444.


 \bibitem{liangli} J. Li; Z. Liang, Some uniform estimates and large-time
behavior of solutions to one-dimensional
compressible Navier-Stokes system in unbounded domains with large data,  Arch. Rational Mech. Anal., \textbf{220} (2016), 1195-1208.

 \bibitem{lx} J.  Li;  Z. Xin, Global existence of weak solutions to the barotropic compressible Navier-Stokes flows with degenerate viscosities,   http://arxiv.org/abs/1504.06826v2.

     \bibitem{liang1}  Z. Liang, Regularity criterion on the energy conservation for the
compressible Navier-Stokes equations,  P. Roy. Soc. Edinb. A, (2020), 1-18.


\bibitem{p2}  P. Lions,  {Mathematical Topics in Fluid Mechanics,}  Volume 1-2, Oxford
Science Publication, Oxford, (1996,1998).


\bibitem{mv} A. Mellet; A.  Vasseur,   On the barotropic compressible Navier-Stokes equations, Commun. Partial  Differ. Equ., \textbf{32}  (2007),  431-452.

\bibitem{nov} A. Novotn\'{y}; I. Stra\u{s}kraba, Introduction to the Mathematical Theory of Compressible Flow. Oxford Lecture Series in Mathematics and its Applications,  Oxford University Press, Oxford, \textbf{27}, 2004.

    \bibitem{ngu} Q.  Nguyen; P.  Nguyen; B Tang, Energy equalities for compressible Navier-Stokes equations, Nonlinearity, \textbf{32}(11) (2019), 4206-4231.

\bibitem{pw} P. Plotnikov;  W. Weigant, Isothermal Navier--Stokes Equations and Radon Transform, SIAM J. Math. Anal., \textbf{47}(1)  (2015), 626-653.

\bibitem{serr} D. Serre, Solutions faibles globales des quations de Navier-Stokes pour un fluide compressible, C. R.
Acad. Sci. Paris. I Math. \textbf{303}(13) (1986),  639-642.

\bibitem{sz} I.  Stra\u{s}kraba; A. Zlotnik,  Global properties of solutions to 1D-viscous compressible
barotropic fluid equations with density dependent viscosity, Z. angew. Math. Phys., \textbf{54} (2003), 593-607.

\bibitem{yu} A. Vasseur; C. Yu, {Existence of global weak solutions for 3D degenerate compressible Navier-Stokes equations,} Invent. Math., \textbf{206}(3)  (2016), 935-974.



    \bibitem{wang} Y. Ye; Y. Wang; W. Wei, Energy equality in the isentropic compressible Navier-Stokes equations allowing vacuum, arXiv: 2108.09425


\bibitem{yu2}   C. Yu, Energy conservation for the weak solutions of the compressible Navier-Stokes equations,  Arch. Ration.
Mech. Anal., \textbf{225}(3) (2017), 1073-1087.

\bibitem{yang1} T. Yang; Z. Yao; C. Zhu,   Compressible Navier-Stokes equations with density dependent viscosity and vacuum, Commun. Partial Differ. Equ., \textbf{26} (2001),
965-981.
\bibitem{yang2} T. Yang; H. Zhao,  A vacuum problem for the one-dimensional compressible
Navier-Stokes equations with density-dependent viscosity, J. Differ. Equ., \textbf{184} (2002), 163-184.
\bibitem{yang3} T. Yang;   C. Zhu, Compressible Navier-Stokes equations with degenerate viscosity
coefficient and vacuum, Comm. Math. Phys., \textbf{230}  (2002),  329-363.

 \end {thebibliography}
\end{document}